\documentclass[12pt, leqno]{amsart}
\input{amssymb.sty}

\usepackage{amsmath,amsthm}
\usepackage{hyperref}
\usepackage[english]{babel}

\usepackage{mathrsfs,amssymb}

\numberwithin{equation}{section}

\title[Strong solidity of group factors]{Strong solidity of group factors from lattices in $\SO(n,1)$ and $\SU(n,1)$}
\author{Thomas Sinclair}
\address{Thomas Sinclair, Vanderbilt University, 1326 Stevenson Center, Nashville, TN 37240}
\email{thomas.sinclair@vanderbilt.edu}
\date{\today}

\dedicatory{}

\theoremstyle{plain}
\newtheorem{thm}{Theorem}[section]
\newtheorem{cor}[thm]{Corollary}
\newtheorem{lem}[thm]{Lemma}
\theoremstyle{definition}

\newtheorem{defn}[thm]{Definition}

\theoremstyle{remark}
\newtheorem{rem}[thm]{Remark}

\newtheorem*{ackns}{Acknowledgements}
\newcommand{\car}{\curvearrowright}

\newcommand{\G}{\Gamma}
\newcommand{\g}{\gamma}
\newcommand{\e}{\varepsilon}
\newcommand{\phii}{\varphi}

\newcommand{\B}{{\mathbb B}}
\newcommand{\C}{{\mathbb C}}
\newcommand{\F}{{\mathbb F}}

\newcommand{\R}{{\mathbb R}}

\newcommand{\Z}{{\mathbb Z}}

\newcommand{\FF}{{\mathcal F}}

\newcommand{\HH}{{\mathcal H}}

\newcommand{\KK}{{\mathcal K}}
\newcommand{\LL}{{\mathcal L}}
\newcommand{\NN}{{\mathcal N}}
\newcommand{\OO}{{\mathcal O}}
\newcommand{\UU}{{\mathcal U}}
\newcommand{\ZZ}{{\mathcal Z}}

\newcommand{\Ad}{\operatorname{Ad}}

\newcommand{\Aut}{\operatorname{Aut}}

\DeclareMathOperator*{\Lim}{Lim}

\newcommand{\SL}{\operatorname{SL}}

\newcommand{\SO}{\operatorname{SO}}
\newcommand{\SU}{\operatorname{SU}}

\newcommand{\ip}[2]{\langle #1, #2 \rangle}
\providecommand{\abs}[1]{\lvert #1 \rvert}
\providecommand{\nor}[1]{\lVert #1 \rVert}

\begin{document}
\begin{abstract} We show that the group factors $L\G$, where $\G$ is an ICC lattice in either $\SO(n,1)$ or $\SU(n,1)$, $n \geq 2$, are strongly solid in the sense of Ozawa and Popa \cite{ozawapopacartanI}.  This strengthens a result of Ozawa and Popa \cite{ozawapopacartanII} showing that these factors do not have Cartan subalgebras.
\end{abstract}

\maketitle
%%%%%%%%%%%%%%%%%%%%%%%%%%%%%%%%%%%%%%%%%%%%%%%%%%%%%%%%%%%%%%%%%%%%%%

\section*{Introduction}

In their breakthrough paper \cite{ozawapopacartanI}, Ozawa and Popa brought new techniques to bear on the study of free group factors which allowed them to show that these factors possess a powerful structural property, what they called ``strong solidity''. 

\begin{defn}[Ozawa--Popa \cite{ozawapopacartanI}]\label{def:strongsolidity}  A $\rm II_1$ factor $M$ is \textit{strongly solid} if for any diffuse amenable subalgebra $P \subset M$ we have that $\mathcal N_M(P)''$ is amenable.
\end{defn}

\noindent As usual, $\NN_M(P) = \{u \in \UU(M) : u P u^\ast = P\}$ denotes the \textit{normalizer} of $P$ in $M$. It can be seen that every nonamenable $\rm II_1$ subfactor of a strongly solid $\rm II_1$ factor is non-Gamma, prime and has no Cartan subalgebras. Thus, Ozawa and Popa's result broadened and offered a unified approach to the two main results on the structure of free group factors hitherto known: Voiculescu's \cite{voiculescuentropyIII} pioneering result, which showed that the free group factors $L\F_n$, $2 \leq n \leq \infty$, have no Cartan subalgebras, and Ozawa's \cite{ozawasolid} seminal work on ``solid'' von Neumann algebras, which showed that every nonamenable $\rm II_1$ subfactor of a free group factor is non-Gamma and prime. Moreover, they exhibited the first, and so far only, examples of $\rm II_1$ factors with a unique Cartan up to unitary conjugacy; namely, the group-measure space constructions of free ergodic profinite actions of groups with property $\rm (HH)^+$ \cite{ozawapopacartanII}; e.g., nonamenable free groups. This improved on the ground-breaking work of Popa \cite{popabetti}, which gave examples of $\rm II_1$ factors with a unique ``HT-Cartan'' subalgebra up to unitary equivalence; e.g., $L(\Z^2 \rtimes \SL_2(\Z))$.

By incorporating ideas and techniques of Peterson \cite{petersonl2}, Ozawa and Popa \cite{ozawapopacartanII} were later able to extend the class of strongly solid factors to, in particular, all group factors of ICC lattices in $\SL(2,\R)$ or $\SL(2,\C)$. Other examples of strongly solid factors were subsequently constructed by Houdayer \cite{houdayerstronglysolid} and by Houdayer and Shlyakhtenko \cite{houdayershlyakhtenko}.

In this paper we will demonstrate the following results:

\begin{thm}\label{thm:main1}  If $\G$ is an ICC lattice in $\mathrm{SO}(n,1)$ or $\mathrm{SU}(n,1)$, $n \geq 2$, then $L\G$ is strongly solid.
\end{thm}

\begin{thm}\label{thm:main2} Let $\G$ be a nonamenable, countable, discrete group which has the complete metric approximation property (Definition \ref{def:cmap}).  If $\G \cong \G_1 \ast \G_2$ decomposes as a non-trivial free product, then $L\G$ has no Cartan subalgebras. Moreover, if $N \subset L\G$ is a $\rm II_1$ factor which has a Cartan subalgebra, then there exists projections $p_1, p_2$ in the center of $N' \cap L\G$ such that $p_1 + p_2 = 1$ and unitaries $u_1, u_2 \in \UU(M)$ such that $u_i N p_i u_i^* \subset L\G_i \subset L\G$, $i \in \{1,2\}$.
\end{thm}

\noindent By a \textit{lattice} we mean a discrete subgroup $\G < G$ of some Lie group with finitely many connected components such that $G/\G$ admits a regular Borel probability measure invariant under left translation by $G$. These factors are already known by the work of Ozawa and Popa \cite{ozawapopacartanII} to have no Cartan subalgebras. Since $\SO(n,1)$ and $\SU(n,1)$ are semisimple Lie groups with finite center, Theorem 6.5 in \cite{cowlinghaagerup} shows that every $\g \in \G$ which is not in the center of $G$ has infinite $\G$-conjugacy class, so examples of ICC lattices abound. In the $\SO(n,1)$ case, the restriction of the lattice subgroup $\G$ to the identity component $\SO_0(n,1)$ is always ICC, $\SO_0(n,1)$ having trivial center, and all results in our paper will hold for these groups.

The proof follows the same strategy as Ozawa and Popa's in \cite{ozawapopacartanI, ozawapopacartanII}. Though, instead of working with closable derivations, we use a natural one-parameter family of deformations first constructed by Parthasarathy and Schmidt \cite{parthaschmidt}. The derivations Ozawa and Popa consider appear as the infinitesimal generators of these deformations (so, the approaches are largely equivalent), but by using the deformations we avoid some of the technical issues which arise when working with derivations. 

The main difficulty in obtaining Theorem \ref{thm:main1} for lattice factors in $\SO(n,1)$ or $\SU(n,1)$ when $n \geq 3$ is that the bimodules which admit good deformations/derivations are themselves too weak to allow one to deduce the amenability of the normalizer algebra e.g., strong solidity.  However, sufficiently large tensor powers of these bimodules can be used to deduce strong solidity. Unfortunately, derivation techniques perturb the original bimodules slightly, and the behavior of tensor powers of the perturbed bimodules becomes unclear. To circumvent this problem, we first notice that Ozawa and Popa's techniques actually allow one to deduce a kind of relative amenability of the normalizer subalgebra with respect to the bimodule, given in terms of an ``invariant mean''. We then use a result of Sauvageot \cite{sauvageottensor} to obtain from the invariant mean an almost invariant sequence of vectors in the bimodule. Since the property of having an almost invariant sequence of vectors is stable under taking tensor powers, we are able to transfer relative amenability to a large tensor power of the bimodule in order to deduce amenability of the normalizer algebra.

\begin{ackns} The author is grateful to Jesse Peterson for suggesting the problem and making several helpful comments and suggestions. The author also thanks Jesse Peterson and Ionut Chifan for many stimulating discussions on the work of Ozawa and Popa. \end{ackns}

\section{Preliminaries}  We collect in this section the necessary definitions, concepts and results needed for the proofs of Theorems \ref{thm:main1} and \ref{thm:main2}.

\subsection{Representations, Correspondences and Weak Containment}
Let $\G$ be a countable discrete group and $\pi,\rho$ be unitary representations of $\G$ into separable Hilbert spaces $\HH_\pi$ and $\HH_\rho$, respectively.

\begin{defn} We say that $\rho$ is \textit{weakly contained} in $\pi$ if for any $\e >0$, $\xi \in \HH_\rho$ and any finite subset $F \subset \G$, there exist vectors $\xi_1',\dots,\xi_n' \in \HH_\pi$ such that $\abs{\ip{\rho(\g)\xi}{\xi} - \sum_{i=1}^n \ip{\pi(\g)\xi_i'}{\xi_i'}} < \e$ for all $\g \in F$.
\end{defn}

\noindent A representation $\pi$ is said to be \textit{tempered} if it is weakly contained in the left-regular representation, and \textit{strongly $\ell^p$} \cite{shalomrigidity} if for any $\e>0$, there exists a dense subspace $\HH_0 \subset \HH$ such that for all $\xi,\eta \in \HH_0$ the matrix coefficient $\ip{\pi(\g)\xi}{\eta}$ belongs to $\ell^{p+\e}(\G)$.  By a theorem of Cowling, Haagerup and Howe \cite{cowlinghaageruphowe}, a representation which is strongly $\ell^2$ is tempered. As was pointed out in \cite{shalomrigidity}, applying standard H\"older estimates to the matrix coefficients, we obtain that if $\pi$ is strongly $\ell^p$ for some $p \geq 2$, then for all $n > p/2$, $\pi^{\otimes n}$ is strongly $\ell^2$, hence tempered.

In the theory of von Neumann algebras, correspondences (also called Hilbert bimodules) play an analogous role to unitary representations in the theory of countable discrete groups. For von Neumann algebras $N$ and $M$, recall that an $N$-$M$ correspondence is a $\ast$-representation $\pi$ of the algebraic tensor $N \odot M^o$ into the bounded operators on a Hilbert space $\HH$ which is normal when restricted to both $N$ and $M^o$.  We will denote the restrictions of $\pi$ to $N$ and $M^o$ by $\pi_N$ and $\pi_{M^o}$, respectively. When the $N$-$M$ correspondence $\pi$ is implicit for the Hilbert space $\HH$, we will use the notation $x\xi y$ to denote $\pi(x \otimes y^o)\xi$, for $x \in N$, $y \in M$ and $\xi \in \HH$.

\begin{defn} Let $\pi: N \odot M^o \rightarrow \B(\HH_\pi)$, $\rho: N \odot M^o \rightarrow \B(\HH_\rho)$ be correspondences.  We say that $\rho$ is \textit{weakly contained} in $\pi$ if for any $\e > 0$, $\xi \in \HH_\rho$ and any finite subsets $F_1 \subset N, \ F_2 \subset M$, there exist vectors $\xi_1,\dotsc,\xi_n' \in \HH_\pi$ such that $\abs{\ip{x\xi y}{\xi} - \sum_{i=1}^n \ip{x\xi_i' y}{\xi_i'}} < \e$ for all $x \in F_1, \ y \in F_2$.
\end{defn}

There is a well-known functor from the category whose objects are (separable) unitary representations of $\G$ and morphisms weak containment to the one of $L\G$-$L\G$ correspondences and weak containment, cf. \cite{popavaeswstar}, which translates the representation theory of $\G$ into the theory of $L\G$-$L\G$ correspondences.  The construction is as follows.  Given $\pi: \G \rightarrow \UU(\HH_\pi)$ a unitary representation, let $\HH^\pi$ be the Hilbert space $\HH_\pi \otimes \ell^2\G$.  Then, the maps $u_\g(\xi \otimes \eta) = \pi(\g)\xi \otimes u_\g\eta$, $(\xi \otimes \eta)u_\g = \xi \otimes (\eta u_{\g})$ extend to commuting normal representations of $L\G$ and $(L\G)^o$ on $\HH^\pi$: the former by Fell's absorption principle, the latter trivially.  This functor is well-behaved with respect to tensor products; i.e., $\HH^{\pi \otimes \rho} \cong \HH^\pi \otimes_{L\G} \HH^\rho$ as $L\G$-$L\G$ correspondences for any unitary $\G$-representations $\pi$ and $\rho$. We refer the reader to \cite{adamenable, popacorrespondences} for the theory of tensor products of correspondences and the basic theory of correspondences in general.

For a $\rm II_1$ factor $M$ there are two canonical correspondences: the \textit{trivial correspondence}, $L^2(M)$ with $M$ acting by left left and right multiplication, and the \textit{coarse correspondence}, $L^2(M) \otimes L^2(\bar M)$ with $M$ acting by left multiplication of the left copy of $L^2(M)$ and right multiplication on the right copy. When $M = L\G$ for some countable discrete group, the trivial and coarse correspondences are the correspondences induced respectively by the trivial and left regular representations of $\G$.

\subsection{Cocycles and the Gaussian Construction}\label{sec:gaussian}  In this section, $\HH$ will denote a \textit{real} Hilbert space which we will fix along with a orthogonal representation $\pi: \G \rightarrow \OO(\HH)$ of some countable discrete group $\G$.

\begin{defn} A cocycle is a map $b : \G \rightarrow \HH$ satisfying the cocycle relation \[b(\g\g') = b(\g) + \pi(\g)b(\g') \textup{, for all } \g,\g' \in \G.\]
\end{defn}

Given $\pi$ and $\HH$, there is a canonical standard probabilty space $(X,\mu)$ and a canonical measure-preserving action $\G \car^\sigma (X,\mu)$ such that there is a Hilbert space embedding of $\HH$ into $L_\R^2(X,\mu)$ intertwining $\pi$ and the natural representation induced on $L_\R^2(X,\mu)$  by $\sigma$.  This is know as the \textit{Gaussian construction}, cf. \cite{petersonsinclair} or \cite{schmidtgaussian}.  It is well-known that the natural $\G$ representation $\sigma_0$ on $L_0^2(X,\mu) = L^2(X,\mu) \ominus \C \, 1_X$ inherits all ``stable'' properties from $\pi$, cf. \cite{petersonsinclair}.  In particular, $\sigma_0^{\otimes n}$ is tempered if and only if $\pi^{\otimes n}$ is tempered for any $n \geq 1$.

It was discovered by Parthasarathy and Schmidt \cite{parthaschmidt} that cocycles also fit well into the framework of the Gaussian construction, inducing one-parameter families of deformations (i.e., cocycles) of the action $\sigma$.  To be precise:

\begin{thm}[Parthasarathy--Schmidt \cite{parthaschmidt}] \label{thm:cocyclefamily}\label{thm:gaussian} Let $b: \G \rightarrow \HH$ be a cocycle, then there exists a one-parameter family $\omega_t: \G \rightarrow \UU(L^\infty(X,\mu)), \ t \in \R$ such that:
\begin{equation}
\omega_t(\g\g') = \omega_t(\g)\sigma_\g(\omega_t(\g')) \textup{, for all } \g,\g' \in \G \textup{; and}
\end{equation}
\begin{equation}\label{eq:cocyclenorm}
\int \omega_t(\g) d\mu = \exp(-(t\nor{b(\g)})^2) \textup{, for all } \g \in \G.
\end{equation}
\end{thm}

\subsection{Weak Compactness and the CMAP}

\begin{defn}[Ozawa--Popa \cite{ozawapopacartanI}]\label{def:weaklycompact} Let $(P,\tau)$ be a finite von Neumann algebra equipped with a trace $\tau$, and $G \car^\sigma P$ be an action of a group $G$ on $A$ by $\tau$-preserving $\ast$-automorphisms.  We say that the action $\sigma$ is \textit{weakly compact} if there exists a net of unit vectors $(\eta_k) \in L^2(P \bar \otimes \bar P, \tau \otimes \bar \tau)_{+}$ such that:
\begin{enumerate}
\item $\nor{\eta_k - (v \otimes \bar v)\eta_k} \rightarrow 0 \text{, for all } v \in \UU(P)$;
\item $\nor{\eta_k - (\sigma_g \otimes \bar \sigma_g)\eta_k} \rightarrow 0 \text{, for all } g \in G$; and
\item $\ip{(x \otimes 1)\eta_k}{\eta_k} = \tau(x) = \ip{(1 \otimes \bar x)\eta_k}{\eta_k} \text{, for all } x \in P$.
\end{enumerate}
\end{defn}

\begin{defn}\label{def:cmap}  A $\rm II_1$ factor $M$ is said to have the \textit{complete metric approximation property} (CMAP) if there exists a net $(\phii_i)$ of finite-rank, normal, completely bounded maps $\phii_i: M \rightarrow M$ such that $\limsup\nor{\phii_i}_{\rm cb} \leq 1$ and such that $\nor{\phii_i(x) - x}_2 \rightarrow 0 \text{, for all } x \in M.$
\end{defn}

\noindent If $\G$ is an ICC countable discrete group, then $L\G$ has the CMAP if and only if the Cowling-Haagerup constant of $\G$, $\Lambda_{\rm cb}(\G)$, equals 1, and if $\G$ is a lattice in $G$, then $\Lambda_{\rm cb}(\G) = \Lambda_{\rm cb}(G)$, cf. \S 12.3 of \cite{brownozawa} and \cite{haagerupwithoutcmap}.

\begin{thm}[Ozawa--Popa, Theorem 3.5 of \cite{ozawapopacartanI}]\label{thm:weaklycompactnormalizers} Let $M$ be a $\rm II_1$ factor which has the CMAP.  Then for any diffuse amenable $\ast$-subalgebra $A \subset M$, $\NN_M(A)$ acts weakly compactly on $A$ by conjugation.
\end{thm}

\section{Amenable Correspondences}

\begin{defn}[Anantharaman-Delaroche \cite{adamenable}] An $N$-$M$ correspondence $\HH$ is called \textit{(left) amenable} if $\HH \otimes_M \bar \HH$ weakly contains the trivial $N$-$N$ correspondence.
\end{defn}

The concept of amenability for correspondences is the von Neumann algebraic analog of the concept of amenablity of a unitary representation of a locally compact group due to Bekka \cite{bekkaamenable}. As was observed by Bekka, amenabilty of the representation $\pi$ is equivalent to the existence of a state $\Phi$ on $\B(\HH)$ satisfying $\Phi(\pi(g)T) = \Phi(T\pi(g))$ for all $g \in G,\ T \in \B(\HH)$.  One can ask if a similar criterion holds for amenable correspondences.  When $M$ is a $\rm II_1$ factor, we will show that this indeed is the case if we replace $\B(\HH)$ with the von Neumann algebra $\NN = \B(\HH) \cap \pi_{M^o}(M^o)'$.  That is, we obtain the following characterization of amenable correspondences:

\begin{thm}[Compare with Theorem 2.1 in \cite{ozawapopacartanI}.]\label{thm:hypertrace}  Let $\HH$ be an $N$-$M$ correspondence with $N$ finite with normal faithful trace $\tau$ and $M$ a $\rm II_1$ factor.  Let $P \subset N$ be a von Neumann subalgebra. Then the following are equivalent:

\begin{enumerate}

\item there exists a net $(\xi_n)$ in $\HH \otimes_M \bar\HH$ such that $\ip{x\xi_n}{\xi_n} \rightarrow \tau(x)$ for all $x \in N$ and $\nor{[u,\xi_n]} \rightarrow 0$ for all $u \in \UU(P)$;

\item there exists a $P$-central state $\Phi$ on $\NN$ such that $\Phi$ is normal when restricted to $N$ and faithful when restricted to $\ZZ(P' \cap N)$;

\item there exists a $P$-central state $\Phi$ on $\NN$ which restricts to $\tau$ on $N$.
\end{enumerate}
\end{thm}

\noindent We do so by constructing a normal, faithful, semi-finite (tracial) weight $\bar\tau$ on $\NN$ which canonically realizes $\HH \otimes_M \bar \HH$ as $\LL^2(\NN,\bar\tau)$.  An identical construction to the one we propose has already appeared in the work of Sauvageot \cite{sauvageottensor} for an arbitrary factor $M$. However, we present an elementary approach in the $\rm II_1$ case.

Recall that a vector $\xi \in \HH$ is \textit{right bounded} if there exists $C >0$ such that for all $x \in M$, $\nor{\xi x} \leq C\nor{x}_2$.  The right-bounded vectors form a dense subspace of $\HH$ which we will denote by $\HH_b$.  Regarding $\HH$ as a right Hilbert $M$-module, we can define a natural $M$-valued inner product on $\HH_b$, which we will denote $(\xi | \eta) \in M$ for $\xi, \eta \in \HH_b$, by setting $(\xi|\eta)$ to be the Radon-Nikodym derivative of the normal functional $x \mapsto \ip{\xi x}{\eta}$.  Then it is easy to see that $(\cdot|\cdot)$ satisfies the following properties for all $\xi,\eta \in \HH_b,\ x,y \in M$: 
\begin{enumerate}
\item $(\xi|\xi) \geq 0$,
\item $(\eta | \xi) = (\xi |\eta)^*$,
\item $(\xi x | \eta y) = y^*(\xi | \eta)x$ 
\end{enumerate}

\noindent (i.e., $(\HH_b, (\cdot|\cdot))$ is an \textit{$M$-rigged space} in the sense of Rieffel \cite{rieffelmorita}). Trivially, we have that $\ip{\xi}{\eta} = \tau((\xi|\eta))$.  Also, by the non-degeneracy of $\pi_{M^o}$, $(\xi|\xi) = 0$ only if $\xi = 0$. Moreover, $(\cdot|\cdot)$ satisfies a noncommutative Cauchy-Schwartz inequality:
\begin{enumerate}
\item[(4)] $(\eta|\xi)(\xi|\eta) \leq \nor{(\xi|\xi)}_\infty (\eta|\eta)$
\end{enumerate}
(cf. \cite{rieffelinduced}, Proposition 2.9).

Let $\NN = \B(\HH) \cap \pi_{M^o}(M^o)'$. For $\xi,\eta \in \HH_b$, let $T_{\xi,\eta}: \HH_b \rightarrow \HH_b$ be the ``rank-one operator'' given by $T_{\xi,\eta}(\, \cdot \,) = \xi(\, \cdot \,|\eta)$.  Then $T_{\xi,\eta}$ extends to a bounded operator with $\nor{T_{\xi,\eta}}_\infty \leq \nor{(\xi|\xi)}_\infty \nor{(\eta|\eta)}_\infty$ \cite{rieffelinduced}.  Notice that $T_{\xi,\xi} \geq 0$ and that $T_{\xi,\xi}$ is a projection if $(\xi|\xi) \in \mathcal P(M)$. Since $T_{\xi,\eta} \, \pi_{M^o}(x) = \pi_{M^o}(x) \, T_{\xi,\eta}$ for all $x \in M^o$, we have that $T_{\xi,\eta} \in \NN$.  It is easy to see that the span of all such operators $T_{\xi,\eta}$ is a $\ast$-subalgebra of $\B(\HH)$ which we will denote by $\NN_f$.  Noticing that for any $S \in \NN$, $S(\HH_b) \subset \HH_b$, we have that $S \, T_{\xi,\eta} = T_{S\xi,\eta}$ and $ T_{\xi,\eta} \, S = T_{\xi, S^\ast\eta}$. It follows that $\NN_f$ is an ideal of $\NN$ which can be considered as the analog of the finite-rank operators in $\B(\HH)$. The following lemma further cements this analogy.

\begin{lem}\label{thm:density}  If $M$ is a $\rm II_1$ factor, we have that $\NN_f' \cap \B(\HH) = \pi_{M^o}(M^o)$; hence, $\NN_f + \C \, 1_{\B(\HH)}$ is weakly dense in $\NN$.
\end{lem}

\begin{proof} The inclusion $\pi_{M^o}(M^o) \subset \NN_f' \cap \B(\HH)$ is trivial.  Conversely, let $T \in \NN_f' \cap \B(\HH)$ and choose a non-zero $\zeta \in \HH_b$ such that $(\zeta|\zeta) = p \in \mathcal P(M)$. (One can always find such a $\zeta$ as $\HH$ has an orthonormal basis of right bounded vectors as a right Hilbert $M$-module.) Choose a sequence $\eta_i \in \HH_b$ such that $\nor{\eta_i - T\zeta} \rightarrow 0$ and let $y_i = (\eta_i|\zeta)$.  Then for every $\xi \in \HH_b$ we have that  $\nor{T(\xi p) - \xi y_i} \leq \nor{T_{\xi,\zeta}}_\infty \nor{\eta_i - T\zeta} \leq \nor{(\xi|\xi)}_\infty \nor{\eta_i - T\zeta}$; so, the sequence $(y_i)$ is $\nor{\cdot}_\infty$-bounded and has a strong cluster point.  Hence, there exists $y \in M$ such that $T(\xi p) = \xi y$ for all $\xi \in \HH$.  Since $M$ is a ${\rm II_1}$ factor, by repeating the argument with $\zeta' = \zeta u$ for $u \in \UU(M)$ and using standard averaging techniques, we conclude that there exists $y_T \in M$ such that $T\xi = \xi y_T$ for all $\xi \in \HH$.  Thus $T = \pi_{M^o}(y_T^o)$.
\end{proof}

Now, consider an element $\phii \in M_\ast$, and define a functional $\bar \phii \in (\NN_f)^\ast$ by $\bar\phii(T_{\xi,\eta}) = \phii((\xi|\eta))$.  It is easy to see that $\bar\phii$ is normal on $\NN_f$ and so, by the preceding lemma, may be extended to a normal semi-finite weight on $\NN$.  Hence, we may construct for each such $\phii$ a noncommutative $L^p$-space over $\NN$, $\LL^p(\NN,\bar\phii) = \{T \in \NN : \nor{T}_p = \bar\phii(\abs{T}^p)^{1/p} < \infty\}$.  If $M$ is a $\rm II_1$ factor with trace $\tau$, then $\bar\tau$ is a normal, faithful, semi-finite trace on $\NN$ and we denote $\LL^p(\NN,\bar\tau)$ simply by $\LL^p(\NN)$. In the case of $\LL^2(\NN)$, we compute that $\nor{T_{\xi,\eta}}_2 = \tau((\xi|\xi)(\eta|\eta)) = \ip{\xi(\eta|\eta)}{\xi}$.  This shows that the map which sends $T_{\xi,\eta}$ to the elementary $M$-tensor $\xi \otimes_M \bar \eta \in \HH \otimes_M \bar \HH$ extends to an $N$-$N$ bimodular Hilbert space isometry from $\LL^2(\NN)$ to $\HH \otimes_M \bar \HH$. We are now ready to prove the motivating result in this section.

\begin{proof}[Proof of Theorem \ref{thm:hypertrace}]  The proof of (1) $\Leftrightarrow$ (3) follows the usual strategy.  For (1) $\Rightarrow$ (3), we have that there exists a net $(\xi_n)$ of vectors in $\HH \otimes_M \bar \HH$ such that $\ip{x\xi_n}{\xi_n} \rightarrow \tau(x)$ for all $x \in N$ and $\nor{[u,\xi_n]} \rightarrow 0$ for all $u \in \UU(P)$. Viewing $\xi_n$ as an element of $\LL^2(\NN)$, let $\Phi_n \in \NN_\ast$ be given by $\Phi_n(T) = \bar\tau(\xi_n^\ast\xi_n \, T)$ for any $T \in \NN$.  Then, by the generalized Powers-St\o rmer inequality (Theorem IX.1.2 in \cite{takesakiii}), we have that $\abs{\Phi_n(x) - \tau(x)} \rightarrow 0$ for all $x \in N$ and $\nor{\Ad(u)\Phi_n - \Phi_n}_1 \rightarrow 0$ for all $u \in \UU(P)$.  Taking a weak cluster point of $(\Phi_n)$ in $\NN^\ast$ gives the required $N$-tracial $P$-central state on $\NN$.  Conversely, given such a state $\Phi$, we can find a net $(\eta_n)$ in $\LL^1(\NN)_+$ such that $\Phi_n(T) = \bar\tau(\eta_n T)$ weakly converges to $\Phi$. In fact, by passing to convex combinations we may assume $\nor{[u,\eta_n]}_1 \rightarrow 0$ for all $u \in \UU(P)$. By another application of the generalized Powers-St\o rmer inequality, it is easy to check that $\xi_n = \eta_n^{1/2} \in \LL^2(\NN) \cong \HH \otimes_M \bar \HH$ satisfies the requirements of (1).

We now need only show (2) $\Rightarrow$ (3) as (3) $\Rightarrow$ (2) is trivial.  But this is exactly the averaging trick found in the proof of (2) $\Rightarrow$ (1) in Theorem 2.1 of \cite{ozawapopacartanI}. We repeat the argument here for the sake of completeness.  Since $\Phi$ is normal on $N$, we have that for some $\eta \in \LL^1(N)_+$, $\Phi(x) = \tau(\eta x)$ for all $x \in N$.  In fact, $\eta \in \LL^1(P' \cap N)_+$ since $\Phi$ is $P$-central.  Denoting by $\FF$ the net of finite subsets of $\UU(P' \cap N)$ under inclusion, for any $\e > 0$, we set $\eta_F = \frac{1}{\abs{F}} \sum_{u \in F} u \eta u^\ast$ and $\xi_{F,\e} = (\chi_{[\e,\infty)}(\eta_F))^{1/2}$.  We now let $\Psi_{F,\e}(T) = \frac{1}{\abs{F}}\sum_{u \in F} \Phi(u \xi_{F,\e} T \xi_{F,\e} u^\ast)$.  Note that $\Psi_{F,\e}$ is still $P$-central.  Now it is easy to see that $\lim_{\FF, \e}\chi_{[\e,\infty)}(\eta_F) = z$, where $z$ is the central support of $\eta$.  But by the faithfulness of $\Phi$ on $\ZZ(P' \cap N)$, we see that $z = 1$.  Hence, any weak cluster point of $(\Psi_{F,\e})_{\FF,\e}$ in $\NN^\ast$ is a $P$-central state which when restricted to $N$ is $\tau$.
\end{proof}

\begin{cor}[generalized Haagerup's criterion for amenability]\label{cor:haagerup} Let $N$, $M$ be $\rm II_1$ factors and $\HH$ an $N$-$M$ correspondence.  If $P \subset N$ is a von Neumann subalgebra, then $\HH$ is left amenable over $P$ (in the sense of Theorem \ref{thm:hypertrace}) if and only if for every non-zero projection $p \in \ZZ(P' \cap N)$ and finite subset $F \subset \UU(P)$, we have \[\nor{\sum_{u \in F} up \otimes \overline{up}}_{\HH \otimes_M \bar \HH,\infty} = \abs{F},\] where $\nor{\,\cdot\,}_{\HH \otimes_M \bar \HH, \infty}$ denotes the operator norm on $\B(\HH\otimes_M\bar\HH)$.
\end{cor}

\begin{proof} Since we have obtained a ``hypertrace'' characterization of amenability for correspondences in Theorem \ref{thm:hypertrace}, the result follows by the same arguments as in Lemma 2.2 in \cite{haagerupinjectivity}.
\end{proof}

\begin{defn}[cf. Definition 1.3 in \cite{petersonsinclair}] Let $M$ be a $\rm II_1$ factor, $\HH$ an $M$-$M$ correspondence and $P^c$ the orthogonal projection onto the subspace $\HH^c = \{\xi \in \HH : x\xi = \xi x,\ \forall x \in M\}$. The correspondence $\HH$ has \textit{spectral gap} if for every $\e > 0$ there exist $\delta > 0$ and $x_1,\dotsc,x_n \in M$ such that if $\nor{x_i \xi - \xi x_i} < \delta$, $i = 1,\dotsc,n$, then $\nor{\xi - P^c\xi} < \e$.  The correspondence $\HH$ has \textit{stable spectral gap} if $\HH \otimes_M \bar\HH$ has spectral gap.
\end{defn}

Note that if $\HH$ has stable spectral gap, then $\HH$ is amenable if and only if $(\HH\otimes_M \bar\HH)^c \neq \{0\}$.  Hence, we say an $M$-$M$ correspondence $\HH$ is \textit{nonamenable} if it has stable spectral gap and $(\HH \otimes_M \bar\HH)^c = \{0\}$. The following theorem is the analog of Lemma 3.2 in \cite{popaspectralgap} for the category of correspondences. N.B. Stable spectral gap as defined in \cite{popaspectralgap} corresponds to our definition of nonamenability.

\begin{thm}\label{thm:nonamenable} Let $M$ be a $\rm II_1$ factor and $\HH$ an $M$-$M$ correspondence.  Then $\HH$ is nonamenable if and only if $\HH \otimes_M \bar\KK$ has spectral gap and for any $M$-$M$ correspondence $\KK$.
\end{thm}

\begin{proof} Let $\HH_b$, $\KK_b$ denote subspaces of right-bounded vectors in $\HH$ and $\KK$, respectively.  Given $\xi \in \HH_b$ and $\eta \in \KK_b$, by the same arguments as above we can define a bounded operator $T_{\xi,\eta} : \KK \rightarrow \HH$ by $T_{\xi,\eta}(\, \cdot \,) = \xi(\, \cdot \,|\eta)$.  As above, one may check that $\nor{(T^*T)^{1/2}}_2 = \nor{\xi \otimes_M \bar\eta} = \nor{(TT^*)^{1/2}}_2$ so that $\HH \otimes_M \KK$ is isometric to a Hilbert-normed subspace of the bounded right M-linear operators from $\HH$ to $\KK$, which we denote $\LL^2(\HH,\KK)$.  Moreover, this identification is natural with respect to the $M$-$M$ bimodular structure on $\LL^2(\HH,\KK)$ given by $x \, T_{\xi,\eta} \, y = T_{x\xi,y^*\eta}$.

We need now only prove the forward implication, as the converse is trivial. Let us fix some arbitrary $M$-$M$ correspondence $\KK$. From Proposition 1.4 in \cite{petersonsinclair}, we have that $(\HH \otimes_M \bar\HH)^c = \{0\}$ if and only if $(\HH \otimes_M \bar\KK)^c = \{0\}$. So, by way of contradiction, we may assume that for every $\e > 0$ and $x_1,\dots,x_n \in M$, there exists a unit vector $\xi \in \HH \otimes_M \bar\KK$ such that $\nor{x_i \xi - \xi x_i}_2 \leq \e$, $i = 1,\dotsc,n$. Without loss of generality, we may assume $x_1,\dots,x_n$ are unitaries. Viewing $\xi$ as an element of $\LL^2(\HH,\KK)$, let $\eta = (\xi^*\xi)^{1/2} \in \LL^2(\HH)$. By the generalized Powers-St\o rmer inequality, we have $\nor{x_i \eta x_i^* - \eta}_2^2 \leq 2 \nor{x_i \xi x_i^* - \xi}_2 \leq 2\e$, $i = 1,\dotsc,n$. Hence, $\eta \in \LL^2(\HH) \cong \HH \otimes_M \bar\HH$ is a unit vector such that $\nor{x_i \eta - \eta x_i}_2 \leq \sqrt{2\e}$, $i = 1,\dotsc,n$. Thus, $\HH \otimes_M \bar\HH$ does not have spectral gap, a contradiction.
\end{proof}

\section{Proofs of Main Theorems}
In this section we prove our main result, from which will follow Theorems \ref{thm:main1} and \ref{thm:main2}. To begin, let $\G$ be an ICC countable discrete group which admits an unbounded cocycle $b : \G \rightarrow \KK$ for some orthogonal representation $\pi : \G \rightarrow \OO(\KK)$.  Let $\G \car^\sigma (X,\mu)$ be the Gaussian construction associated to $\pi$ as described in section \ref{sec:gaussian} and $\{\omega_t : t \in \R\}$ be the one-parameter family of cocycles associated to $b$ as given by Theorem \ref{thm:cocyclefamily}. Let $\alpha_t$ be the $\ast$-automorphism of $\tilde M = L^\infty(X,\mu) \rtimes \G$ defined by $\alpha_t(a u_\g) = a \omega_t(\g) u_\g$ for all $a \in L^\infty(X,\mu), \ \g \in \G$. Finally, we set $M = L\G$, and we denote by $\HH$ the $M$-$M$ bimodule $L_0^2(X,\mu) \otimes \ell^2\G$ with the usual bimodule structure; i.e., the one defined by $u_\g(\xi \otimes \eta) = \sigma(\g)\xi \otimes u_\g\eta$, $(\xi \otimes \eta)u_\g = \xi \otimes (\eta u_\g)$ for all $\xi \in L_0^2(X,\mu), \ \eta \in \ell^2\G$ and $\g \in \G$.

\begin{thm}\label{thm:amenablenormalizer}  With the assumptions and notations as above, suppose $P \subset M$ is a diffuse von Neumann subalgebra such that $\NN_M(P)$ acts weakly compactly on $P$ via conjugation. Let $Q = \NN_M(P)''$. If either: (1) $b$ is a proper cocycle; or (2) $\pi$ is a mixing representation and $\alpha_t$ does not converge $\nor{\cdot}_2$-uniformly to the identity on $(Qp)_1$ for any projection $p \in \ZZ(Q' \cap M)$ as $t \rightarrow 0$, then the $M$-$M$ correspondence $\HH$ is left amenable over $Q$ in the sense of satisfying Theorem \ref{thm:hypertrace} for $Q \subset M$.
\end{thm}

\begin{proof} In the case of (1), since $b$ is proper, it is easy to see by formula \ref{eq:cocyclenorm} that $E_{L\G} \circ \alpha_t$ restricted to $L\G$ is compact for all $t > 0$. Hence, by the proof of Theorem 4.9 in \cite{ozawapopacartanI}, for any non-zero projection $p \in \ZZ(Q' \cap M)$ and any finite subset $F \subset \NN_M(P)$, we can find a vector $\xi_{p,F} \in \HH \otimes L^2(\bar M)$ such that $\nor{x \xi_{p,F}} \leq \nor{x}_2$ for all $x \in M$, $\nor{p \xi_{p,F}} \geq \nor{p}_2/8$ and $\nor{[u \otimes \bar u, \xi_{p,F}]} < 1/\abs{F}$ for all $u \in F$.

In the case of (2), we need only demonstrate that our assumptions imply the existence of such a net $(\xi_{p,F})$ as in case (1) then argue commonly for both sets of assumptions. In this direction, it is suffient to show that under the assumptions of (2), $\alpha_t$ does not converge uniformly on $(Pp)_1$ for any projection $p \in \ZZ(Q' \cap M)$. Then, by contradiction if such a net $(\xi_{p,F})$ did not exist, the proof of Theorem 4.9 in \cite{ozawapopacartanI} shows that for some $t >0$ sufficiently small we have $\nor{E_M \circ \alpha_t(up)}_2 \geq \nor{p}_2/4$ for all $u \in \UU(P)$.  But since $\alpha_t$ is a one-parameter strongly-continuous group of automorphisms, we would have uniform convergence of $\alpha_t$ on $\UU(P)p$, and by Kaplansky's density theorem on $(Pp)_1$, as $t \rightarrow 0$.

To conclude the discussion of case (2), we suppose $\pi$ is mixing and $\alpha_t$ converges $\nor{\cdot}_2$-uniformly on $(P)_1$ and show $\alpha_t$ converges uniformly on $(Q)_1$. Since there is a natural trace-preserving automorphism $\beta \in \Aut(\tilde M)$ which pointwise fixes $M$ and such that $\beta \circ \alpha_t = \alpha_{-t} \circ \beta$ for all $t \in \R$ (cf. \cite{petersonsinclair}), by Popa's transversality lemma \cite{popaspectralgap}, it is enough to show that $\nor{\alpha_t(x) - E_M \circ \alpha_t(x)}_2 \rightarrow 0$ uniformly on $(Q)_1$.  Notice that $\delta_t(x) = \alpha_t(x) - E_M \circ \alpha_t(x) = (1 - E_M)(\alpha_t(x) - x)$ is a (bounded) derivation $\delta_t : M \rightarrow \HH$.  Since $\pi$ is mixing, by \cite{petersonw*e} we have that $\HH$ is a compact correspondence; hence, by Theorem 4.5 in \cite{petersonl2} $\delta_t \rightarrow 0$ uniformly in $\nor{\cdot}_2$-norm on $(Q)_1$ as $t \rightarrow 0$.

Now, proceeding commonly for both cases, let $\FF$ be the net of finite subsets of $\NN_M(P)$ under inclusion. We define the state $\Phi_p$ on $\NN = \B(\HH) \cap \pi_{M^o}(M^o)'$ by \[\Phi_p(T) = \Lim_\FF \frac{1}{\nor{p \xi_{p,F}}_2}\ip{(T \otimes 1)p\xi_{p,F}}{p\xi_{p,F}},\] where $\Lim_\FF$ is an arbitrary Banach limit. It is easy to see by the properties of $\xi_{p,F}$ that $\Phi_p$ is normal on $M$. Proceeding as in Lemma 5.3 in \cite{ozawapopacartanII}, we then have that for all $u \in \NN_M(P)$, 
\begin{equation*}
\begin{split}\Phi_p(u^\ast T u) &= \Lim_\FF \frac{1}{\nor{p \xi_{p,F}}_2}\ip{(T \otimes 1)up\xi_{p,F}}{up\xi_{p,F}} \\
&= \Lim_\FF \frac{1}{\nor{p \xi_{p,F}}_2}\ip{(T \otimes 1)p(u \otimes \bar u)\xi_{p,F}}{p(u \otimes \bar u)\xi_{p,F}} \\
&= \Lim_\FF \frac{1}{\nor{p \xi_{p,F}}_2}\ip{(T \otimes 1)p(u \otimes \bar u)\xi_{p,F}(u \otimes \bar u)^\ast}{p(u \otimes \bar u)\xi_{p,F}(u \otimes \bar u)^\ast} \\
&= \Phi_p(T).
\end{split}
\end{equation*}

\noindent Hence, we have that $\Phi_p([x,T]) = 0$ for all $x$ in the span of $\NN_M(P)$ and $T \in \NN$.  But we have that $\abs{\Phi_p(T x)} \leq \nor{T}_\infty \abs{\Lim_\FF \frac{1}{\nor{p \xi_{p,F}}_2}\ip{x p\xi_{p,F}}{p\xi_{p,F}}} \leq 8 \nor{T}_\infty \nor{x}_2$ and similarly for $\abs{\Phi_p(xT)}$.  Thus, by Kaplansky's density theorem we have that $\Phi_p$ is a $Q$-central state.

To summarize, for every non-zero projection $p \in \ZZ(Q' \cap M)$, we have obtained a state $\Phi_p$ on $\NN$ such that $\Phi_p(p) = 1$, $\Phi_p$ is normal on $M$ and $\Phi_p$ is $Q$-central.  A simple maximality argument then shows that there exists a state $\Phi$ on $\NN$ which is normal on $M$, $Q$-central, and faithful on $\ZZ(Q' \cap M)$.  Thus, $\Phi$ satisfies condition (2) of Theorem \ref{thm:hypertrace}, and we are done.
\end{proof}

Keeping with the same notations, assume now that the orthogonal representation $b : \G \rightarrow \OO(\KK)$ is such that there exists an $K > 0$ such that $\pi^{\otimes K}$ is weakly contained in the left regular representation.  As was pointed out in section \ref{sec:gaussian}, the representation induced on $L_0^2(X,\mu)$ by $\G \car^\sigma (X,\mu)$ also has this property. Let $\HH_\sigma = L_0^2(X,\mu)$ so that $\HH = \HH_\sigma \otimes \ell^2\G$ is the $M$-$M$ correspondence induced by the representation $\sigma$. Denote by $\tilde\HH^n$ the $M$-$M$ correspondence $((\HH_\sigma \otimes \bar\HH_\sigma)^{\otimes n}) \otimes \ell^2\G$ with the natural bimodule structure. It is straightforward to check that $\HH \otimes_M \bar\HH \cong \HH_\sigma \otimes \ell^2\G \otimes \bar{\HH_\sigma}$ and that $(\HH \otimes_M \bar\HH) \otimes_M \cdots \otimes_M (\HH \otimes_M \bar\HH)$ for $n+1$ copies is isomorphic to $\HH \otimes_M (\tilde\HH^n) \otimes_M \bar\HH$ as $M$-$M$ bimodules. Hence, the $M$-tensor product of $K$ copies of $\HH \otimes_M \bar\HH$ is weakly contained in the coarse $M$-$M$ correspondence.

\begin{thm}\label{thm:generalizedmain} With the assumptions and notations as above, suppose $P \subset M$ is a diffuse von Neumann subalgebra such that $\NN_M(P)$ acts weakly compactly on $P$ via conjugation. Then $Q = \NN_M(P)''$ is amenable.
\end{thm}

\begin{proof} Let $p$ be a non-zero projection in $\ZZ(Q' \cap M)$.  By the proofs of Theorems \ref{thm:amenablenormalizer} and \ref{thm:hypertrace}, it follows that we can find a net $(\xi_n)$ in $\HH \otimes_M \bar\HH$ such that $\ip{x \xi_n}{\xi_n} \rightarrow \tau(pxp)/\tau(p)$ for all $x \in M$ and $\nor{[u,\xi_n]} \rightarrow 0$ for all $u \in \UU(Q)$.  In fact, without loss of generality we may assume that $\ip{x \xi_n}{\xi_n} = \tau(pxp)/\tau(p) = \ip{\xi_n x}{\xi_n}$ for all $x \in M$ (see the proof of Proposition 3.2 in \cite{ozawapopacartanI}). In particular, $(\xi_n)$ is uniformly left and right bounded. Let $\tilde\xi_n$ be the $M$-tensor product of $K$ copies of $\xi_n$. Then then net $(\tilde\xi_n)$ may be seen to satisfy the same properties. Since $(\tilde\xi_n)$ are vectors in a correspondence weakly contained in the coarse $M$-$M$ correspondence, we have that for any finite subset $F \subset \UU(Q)$ that \[\nor{\sum_{u \in F} up \otimes \overline{up}}_{M \bar\otimes \bar M} \geq \lim_n\nor{\sum_{u \in F}u \tilde\xi_n u^\ast} = \abs{F}.\]  Hence, by Haagerup's criterion \cite{haagerupinjectivity}, $Q$ is amenable.
\end{proof}

\begin{rem} Let $M$ be a $\rm II_1$ factor and $\delta$ a closable real derivation from $M$ into an $M$-$M$ correspondence $\HH$ (cf. \cite{petersonl2}). Suppose $P \subset M$ is a von Neumann subalgebra and $\NN_M(P)$ acts weakly compactly on $P$ by conjugation.  Let $Q = \NN_M(P)''$.  One can show that if $\delta^*\bar\delta$ has compact resolvents, then $\HH$ is left amenable over $Q$. The proof is identical to the proof of Theorem B in \cite{ozawapopacartanII}, using the generalized Haagerup's criterion (Corollary \ref{cor:haagerup}). In particular, this sharpens Theorem A in \cite{ozawapopacartanII}: any $\rm II_1$ factor $M$ with the CMAP admitting such a derivation into a nonamenable correspondence has no Cartan subalgebras.
\end{rem}

We are now ready to prove Theorems \ref{thm:main1} and \ref{thm:main2}.

\begin{proof}[Proof of Theorem \ref{thm:main1}]
If $\G$ is an ICC lattice in $\SO(n,1)$ for $n \geq 3$ or $\SU(n,1)$ for $n \geq 2$, then Theorems 1.9 in \cite{shalomrigidity} shows that $\G$ possesses an unbounded cocycle into some strongly $\ell^p$ representation for $p \geq 2$. By Theorem 3.4 in the same, any unbounded cocycle for such a lattice is proper.  Since $L\G$ has the CMAP by \cite{decannierehaagerup} and \cite{haagerupwithoutcmap}, by Theorems \ref{thm:weaklycompactnormalizers} and \ref{thm:generalizedmain} the result obtains.
\end{proof}

\begin{proof}[Proof of Theorem \ref{thm:main2}]
For the first assertion, since $\G \cong \G_1 \ast \G_2$, $\G$ admits a canonical unbounded cocycle $b : \G \rightarrow \bigoplus^\infty\ell^2\G$ into a direct sum of left regular representations. The left regular representation is mixing and $L\G$-$L\G$ correspondence associated to the left regular representation (the coarse correspondence) is amenable if and only if $\G$ is amenable. So, if $L\G$ did admit a Cartan subalgebra, then by Theorems \ref{thm:weaklycompactnormalizers} and \ref{thm:amenablenormalizer} the deformation $\alpha_t$ of $L\G$ obtained from the cocycle $b$ would have to converge uniformly on $(L\G)_1$ as $t \rightarrow 0$. But this contradicts that $b$ is unbounded.

For the second assertion, if a $\rm II_1$ subfactor $N \subset L\G$ admits a Cartan subalgebra, then we have that $\alpha_t$ converges $\nor{\,\cdot\,}_2$-uniformly on the unit ball of $N$, since $N$ has the CMAP and the coarse $L\G$-$L\G$ correspondence viewed an $N$-$N$ correspondence embeds into a direct sum of coarse $N$-$N$ correspondences. Let $\tilde\G = \G \ast \F_2$, where $\F_2$ is the free group on two generators.  Let $u_1,u_2 \in L\F_2$ be the canonical generating unitaries and $h_1,h_2 \in L\F_2$ self-adjoint elements such that $u_j = \exp(\pi i h_j)$, $j = 1,2$.  Define $u_j^t = \exp(\pi i t h_j)$, $j = 1,2$, and let $\theta_t$ be the $\ast$-automorphism of $L\tilde\G$ given by $\theta_t = \Ad(u_1^t) \ast \Ad(u_2^t)$. It follows from Lemma 5.1 in \cite{petersonl2} and Corollary 4.2 in \cite{petersonsinclair} that $\theta_t$ converges uniformly in $\nor{\,\cdot\,}_2$-norm on the unit ball of $N \subset L\G \subset L\tilde\Gamma$ as $t \rightarrow 0$. An examination of the proof of Theorem 4.3 in \cite{ioanapetersonpopa} shows that this is the only condition necessary for the theorem to obtain.  Our result then follows directly from Theorem 5.1 in \cite{ioanapetersonpopa}.
\end{proof}

%%%%%%%%%%%%%%%%%%%%%%%%%%%%%%%%%%%%%%%%%%%%%%%%%%%%%%%%%%%%%%%%%%%%%%
\bibliographystyle{amsplain}
\bibliography{bibfile}

\end{document}